\documentclass[11pt]{amsart}

\usepackage{epigamath}

\usepackage[english]{babel}

\numberwithin{equation}{section}

\usepackage{enumitem}
\usepackage{mathrsfs,bm,ddag}
\usepackage[all]{xy}

\newtheorem{thm}[subsection]{Theorem}
\newtheorem{lem}[subsection]{Lemma}
\newtheorem*{lem*}{Lemma}

\newtheorem*{cor*}{Corollary}

\newtheorem*{prop*}{Proposition}

\newtheorem*{conj*}{Conjecture}
\newtheorem*{thm*}{Theorem}

\theoremstyle{definition}

\newtheorem{dfn}[subsection]{Definition}
\newtheorem*{dfn*}{Definition}

\theoremstyle{remark}

\newtheorem*{ex*}{Example}
\newtheorem{rem}[subsection]{Remark}
\newtheorem*{rem*}{Remark}

\newcommand{\SH}{\mc{SH}}

\newcommand{\ul}[1]{\underline{#1}}

\newcommand{\PrLst}{\mc{P}\mr{r}^{\mr{L}}_{\mr{st}}}
\newcommand{\PrL}{\mc{P}\mr{r}^{\mr{L}}}

\newcommand{\rHbm}{\mr{H}^{\mr{BM}}}
\newcommand{\ulrHbm}{\ul{\mr{H}}^{\mr{BM}}}
\newcommand{\sHbm}{\mc{H}^{\mr{BM}}}
\newcommand{\Sp}{\mc{S}\mr{p}}

\newcommand{\Hom}{\mr{Hom}}
\newcommand{\Hc}{\mc{H}_{\mr{c}}^*}
\renewcommand{\H}{\mc{H}}

\newcommand{\Arr}{\widetilde{\mr{Ar}}}
\newcommand{\Shv}{\mc{S}\mr{hv}}
\newcommand{\PShv}{\mc{P}}
\newcommand{\Mod}{\mr{Mod}}
\newcommand{\App}{\mr{App}}
\newcommand{\Tri}{\mc{T}\mr{ri}}

\EpigaVolumeYear{7}{2023} \EpigaArticleNr{7}
\ReceivedOn{June 27, 2022}
\InFinalFormOn{October 4, 2022}
\AcceptedOn{November 18, 2022}

\title{Trace formalism for motivic cohomology}
\titlemark{Trace formalism for motivic cohomology}

\author{Tomoyuki~Abe}
\address{Kavli Institute for the Physics and Mathematics of the Universe (WPI), University of Tokyo, 5-1-5 Kashiwanoha,  Kashiwa, Chiba, 277-8583, Japan}
\email{tomoyuki.abe@ipmu.jp}

\authormark{T.~Abe}

\AbstractInEnglish{A goal of this paper is to construct trace maps for the six functor formalism of motivic cohomology after Voevodsky, Ayoub,
 and Cisinski--D\'{e}glise.
 We also construct an $\infty$-enhancement of such a trace formalism.
 In the course of the $\infty$-enhancement, we need to reinterpret the trace formalism in a more functorial manner.
 This is done by using Suslin--Voevodsky's relative cycle groups.}



\acknowledgement{This work is supported by JSPS KAKENHI Grant Numbers 16H05993, 18H03667, 20H01790.}

\begin{document}



\maketitle

\begin{prelims}

\DisplayAbstractInEnglish

\bigskip

\DisplayKeyWords

\medskip

\DisplayMSCclass







\end{prelims}


\newpage

\setcounter{tocdepth}{1}

\tableofcontents


\section{Introduction}
Let $f\colon X\rightarrow S$ be a flat morphism of dimension $d$ between schemes of finite type over a field $k$.
Let $\Lambda$ be a torsion ring in which the exponential characteristic of $k$ is invertible.
In \cite[Expos\'e~XVIII, Th\'eor\`eme~2.9]{SGA}, the trace map $\mr{Tr}_f\colon \mr{R}f_!f^*\Lambda(d)[2d]\rightarrow\Lambda$
satisfying various functorial properties is constructed.
Here, the cohomological functors are taken for the \'{e}tale topoi.
Furthermore, the trace map is characterized by such functorialities.
This trace map is fundamentally important, and for example, it is used to construct the cycle class map.
In other words, we may view the trace formalism as a device to throw cycle-theoretic information into the cohomological framework.
The main goal of this paper is to construct an analogous map for the motivic cohomology of Voevodsky, and its $\infty$-enhancement.
The $\infty$-enhancement of the trace formalism will serve as an interface between
``actual cycle'' and ``$\infty$-enhancement of motivic cohomology'' in \cite{A2}.

Let us explain the method to construct the trace formalism.
From now on, we consider the six functor formalism of the motivic cohomology theory with coefficients in $\Lambda:=\mb{Z}[1/p]$,
where $p$ is the characteristic of our base field $k$.
The principle that makes the construction of the trace map work is the observation that the higher homotopies vanish.
More precisely, we have
\begin{equation}
 \label{vanhighhom}
 \mr{R}^i\mr{Hom}\bigl(\mr{R}f_!f^*\Lambda(d)[2d],\Lambda\bigr)=0
\end{equation}
for $i<0$.
A benefit of this vanishing is that if we take an open subscheme $j\colon U\hookrightarrow X$ such that $U_s\subset X_s$ is dense for any $s\in S$,
then constructing $\mr{Tr}_f$ and constructing $\mr{Tr}_{f\circ j}$ are equivalent.
In \cite{SGA}, this property is used ingeniously to reduce the construction to simpler situations.
Another benefit which is more important for us is that the vanishing allows us to construct the map ``locally''.
Namely, by the vanishing, constructing $\mr{Tr}_f$ is equivalent to constructing a morphism
$\mr{R}^{2d}f_!f^*\Lambda(d)\rightarrow\Lambda$ of {\em sheaves}.
In the case of \'{e}tale cohomology, since it admits proper descent,
by de Jong's alteration theorem, the construction is reduced to the case where $S$ is smooth.
We note that we commonly use de Jong's alteration theorem to reduce proving {\em properties} to smooth cases,
but to reduce {\em constructions} to smooth cases needs control of higher homotopies, which requires great amount of effort in general.
In the case where $S$ is smooth, the construction is easy because we have an isomorphism
$\mr{Hom}\bigl(\mr{R}f_!f^*\Lambda(d)[2d],\Lambda\bigr)\cong\mr{Hom}\bigl(\mr{R}p_{X!}p_X^*\Lambda(d_X)[2d_X],\Lambda\bigr)$,
where $p_X$ is the structural morphism for $X$ and $d_X:=\dim(X)$, using the relative Poincar\'{e} duality,
namely the isomorphism $p_S^*(d_S)[2d_S]\cong p_S^!$.
In the case of \'{e}tale cohomology, in \cite{SGA}, the relative Poincar\'{e} duality theorem is established by using the trace formalism,
and the argument we explained here is somewhat circular.
However, in the theory of motives, the relative Poincar\'{e} duality follows from theorems of Morel--Voevodsky, Ayoub,
and Cisinski--D\'{e}glise which use completely different methods, and the above argument actually works.

Now, assume we wish to enhance the trace map $\infty$-categorically.
The first question that immediately comes up with is that what it means by ``$\infty$-enhancement'' in this situation.
To address the question, we need a reinterpretion of the trace map, and to motivate our reinterpretation, let us discuss a defect of traditional formalism.
Let $f$ be a flat morphism between {\em non-reduced} schemes such that $f_{\mr{red}}$ is {\em not} flat.
In this situation, we have the trace map $\mr{Tr}_f$.
However, since motivic or \'{e}tale cohomology is insensitive to nil-immersions, $\mr{Tr}_f$ induces a similar map for $f_{\mr{red}}$.
This observation gives us an impression that the trace map should be associated with a ``cycle'' rather than a ``scheme''.
To realize this idea, we use the relative cycle group of Suslin and Voevodsky.
For a morphism $f\colon X\rightarrow S$, they defined a group denoted by $z_{\mr{equi}}(X/S,d)$
which is a certain subgroup of the group of cycles in $X$ equidimensional of dimension $d$ over $S$ (see \cite{SV}).
When $f$ is flat of dimension $d$, the cycle $[X]$ is an element of $z_{\mr{equi}}(X/S,d)$.
Using these observations, we show that there exists a morphism $z_{\mr{equi}}(X/S,n)\rightarrow\mr{Hom}\bigl(\mr{R}f_!f^*\Lambda(n)[2n],\Lambda\bigr)$
for any $n$, such that, when $f$ is flat of dimension $d$, the image of $[X]\in z_{\mr{equi}}(X/S,d)$ is the traditional trace map.
The object $\mr{Hom}\bigl(\mr{R}f_!f^*\Lambda(n),\Lambda\bigr)$ is often called the {\em Borel--Moore homology},
and is denoted by $\rHbm(X/S,\Lambda(n))$.
Note that we are considering it as an object of the derived category (or as a spectrum).
The associations $z_{\mr{equi}}(X/S,n)$ and $\rHbm(X/S,\Lambda(n))$ to $X/S$ are functorial with respect to the base changes of $S$
and pushforwards along proper morphisms $X\rightarrow X'$ over $S$.
These functorialities yield ($\infty$-)functors from a certain category $\widetilde{\mr{Ar}}$ to the $\infty$-category of spectra $\Sp$.
The $\infty$-enhancement of the trace map can be formulated as a natural transform between these $\infty$-functors,
and we will show the existence of such an $\infty$-functor in the last section.
This $\infty$-enhancement of the trace map is one of the crucial ingredients in \cite{A2}.

Before concluding the introduction, let us present the organization of this paper.
In Section~\ref{sec1}, we recall the six functor formalism of the theory of motives after Voevodsky, Ayoub, Cisinski--D\'{e}glise.
In Section~\ref{sec2}, we formulate our main result.
To describe the functoriality of $z_{\mr{equi}}(X/S,n)$ and $\rHbm(X/S,\Lambda(n))$ above,
it is convenient to use the language of ``bivariant theory'' after Fulton--MacPherson.
We start by recalling such a theory, and we state our main theorem.
We conclude this section by showing an analogue of (\ref{vanhighhom}) in the motivic setting.
In Section~\ref{sec3}, we construct the trace map in the case where the base scheme $S$ is smooth.
In Section~\ref{pfmthm}, we construct the trace map in general and show the main result.
In Section~\ref{sec5}, we establish the $\infty$-enhancement.
We note that, even though we use the language of $\infty$-categories throughout the paper for convenience and coherence,
it is straightforward to formulate and prove the results of Sections~\ref{sec1} to~\ref{pfmthm}
using the language of model categories, as in \cite{CD,CDbook}.
Using the language of $\infty$-categories  is more essential in Section~\ref{sec5}.

\subsection*{Acknowledgment}
\leavevmode
\medskip

The author is grateful to Deepam Patel for numerous discussions, without which this paper would not have been written.
He wishes to thank Adeel Kahn for various helpful comments on the paper.
Especially, Remark~\ref{mainrem}-\eqref{mainrem-4} is due to him.
He also thanks Fr\'{e}d\'{e}ric D\'{e}glise for answering several questions, and Shane Kelly for some discussions.
Finally, he wishes to thank the referee for reading the manuscript very carefully, and gave him numerous comments which helped to improve the quality of the paper.

\subsection*{Notation and conventions}
\leavevmode
\medskip

We fix a perfect field $k$ of characteristic $p>0$.
By $\infty$-category, we always mean $(\infty,1)$-category, and by category we always mean $1$-category.
For a scheme $S$, we denote by $\mr{Sch}_{/S}$ the category of schemes separated of finite type over $S$.
When $S=\mr{Spec}(k)$, we often denote this by $\mr{Sch}_{/k}$.

\section{Review of six functors}
\label{sec1}

\subsection{}
\label{fixcoh}
We will use the language of $\infty$-categories, but except for \S\ref{sec5}, this is used just to facilitate the presentation.
See the remark at the end of this paragraph for some explanation.

Let $\PrLst$ be the full subcategory of $\PrL$ (\emph{cf.}~\cite[Definition 5.5.3.1]{HTT}) spanned by stable $\infty$-categories.
We have the functor $\SH\colon\mr{Sch}_{/k}^{\mr{op}}\rightarrow\PrLst$ sending $T$ to Voevodsky-Morel's stable homotopy $\infty$-category $\SH(T)$
(\emph{cf.}~\cite[\S2.1]{CD} or \cite[Example 1.4.3]{CDbook} for model categorical treatment and \cite[\S6.7]{A} and references therein for $\infty$-categorical treatment).
Let $\Lambda$ be a commutative ring.
Then Voevodsky defined the motivic Eilenberg-MacLane spectrum $\mb{H}\Lambda_k$, which is an $\mb{E}_\infty$-algebra of $\SH(\mr{Spec}(k))$.
By pulling back, this spectrum yields a spectrum $\mb{H}\Lambda_{T/k}$ on $\SH(T)$,
and defines an ``absolute ring $\SH$-spectrum'' in the sense of \cite[Definition 1.1.1]{Deg}.
The absolute ring $\SH$-spectrum $\mb{H}\Lambda_{T/k}$ is equipped with an ``orientation'' in the sense of \cite[Definition 2.2.2]{Deg} by \cite[Example 2.2.4]{Deg}.
Under this situation, all the results of \cite[Introduction, Theorem~1]{Deg} can be applied.
We do not try to recall the definitions of each terminology, but instead, we sketch what we can get by fixing these data.

We put $\mc{D}_T:=\Mod_{\mb{H}\Lambda_T}(\SH(T))$, the symmetric monoidal $\infty$-category of $\mb{H}\Lambda_T$-module objects in $\SH(T)$.
Then the assignment $\mc{D}_T$ to $T$ can be promoted to a functor $\mc{D}\colon\mr{Sch}_{/S}^{\mr{op}}\rightarrow\PrLst$
which yields ``motivic categories'' in the sense of \cite{CDbook}.
This can be checked from \cite[Proposition 5.3.1 and Proposition~7.2.18]{CDbook}.
We may find a summary of the axioms of what this means in \cite[\S6.1]{A}, and also references.
Among other things, we may use ``six functors''.
In this $\infty$-categorical context, we can find a construction of six functor formalism in \cite[\S6.8]{A},
which follows the idea of \cite{Khan}.
Let $X\in\mr{Sch}_{/S}$.
Then $\mc{D}_X$ is a symmetric monoidal stable $\infty$-category.
Given a morphism $f\colon X\rightarrow Y$ in $\mr{Sch}_{/S}$, the functor $\mc{D}$ induces the functor $\mc{D}_Y\rightarrow\mc{D}_X$,
which we denote by $f^*$ in accordance with the six functor formalism of Grothendieck.
The functor $f^*$ admits a right adjoint, which we denote by $f_*$.
We also have the ``extraordinary pushforward functor'' $f_!\colon\mc{D}_X\rightarrow\mc{D}_Y$ as well as its right adjoint $f^!$.
We have the natural transform $f_!\rightarrow f_*$ which is an isomorphism when $f$ is proper.

The orientation on $\mb{H}\Lambda$ yields an orientation on $\mc{D}$ in the sense of \cite[Definition~2.4.12]{CDbook} by \cite[Example~2.4.40]{CDbook} and \cite[\S2.2.5]{Deg}.
For $n\in\mb{Z}$, we denote the $n$-th Tate twist by $(n)$, the $n$-th shift by $[n]$, and $(n)[2n]$ by $\left<n\right>$.
We often denote the unit object of $\mc{D}_T$ by $\Lambda_T$.
By fixing an orientation, we have a canonical isomorphism $f^*(d)[2d]\cong f^!$ for any smooth morphism $f$ in $\mr{Sch}_{/S}$ (\emph{cf.}~\cite[Theorem~2.4.50]{CDbook}).
In fact, the fundamental class constructed in \cite[Introduction, Theorem~1]{Deg} can be seen as a generalization of this isomorphism.

\begin{rem*}
 If the reader feels uncomfortable with using $\infty$-categories,
 it is essentially harmless to replace $\PrLst$ by the the $(2,1)$-category of triangulated categories $\Tri$ above.
 Then, we may regard $\mc{D}_T$ as a triangulated category.
 The only exception might be that when we consider descents.
 In order to consider descents inside the traditional framework, we need to introduce the category of diagrams as in \cite[\S3]{CDbook}.
 Therefore, strictly speaking, simply considering the functor $\mr{Sch}_{/S}^{\mr{op}}\rightarrow\Tri$ is not enough for the theory of descent.
 We leave the details to the interested reader.
\end{rem*}

\subsection{}
Let $f\colon X\rightarrow S$. For $\mc{F}\in\mc{D}_S$, we set
\[
 \Hc(X/S,\mc{F}):=f_!f^*\mc{F},\quad
 \H^*(X/S,\mc{F}):=f_*f^*\mc{F}\quad\text{and}\quad
 \rHbm(X/S,\mc{F}):=\Hom_{\mc{D}_S}
 \bigl(\Hc(X/S,\mc{F}),\Lambda_S\bigr).
\]
Here, we view $\rHbm$ as a spectrum.
When the coefficient ring $\Lambda$ is obvious, we abbreviate $\rHbm(X/S,\Lambda(n))$ by $\rHbm(X/S,n)$.
We write $\rHbm_m(X/S,\mc{F})$ for $\pi_m\rHbm(X/S,\mc{F})$, and call it the {\em Borel-Moore homology}.
Note that $\pi_m\rHbm(X/S,n)$ coincides with $(\mb{H}\Lambda)^{\mr{BM}}_{m,n}(X/S)$ in \cite{Deg}.
Assume we are given a closed subscheme $Z\subset X$ and denote the complement by $U$.
By localization sequence of 6-functor formalism, we have the long exact
sequence
\begin{equation*}
 \cdots\longrightarrow
 \rHbm_m(Z/S,\mc{F})
  \longrightarrow
  \rHbm_m(X/S,\mc{F})
  \longrightarrow
  \rHbm_m(U/S,\mc{F})
  \longrightarrow
  \rHbm_{m-1}(Z/S,\mc{F})
  \longrightarrow\cdots.
\end{equation*}

\subsection{}
\label{GTres}
We introduce the $p$dh-topology as follows.

\begin{dfn*}
 We define {\em $p$dh-topology} on $\mr{Sch}_{/k}$ to be the topology generated by the following two types of families:
 \begin{enumerate}
  \item $\{f\colon Y\rightarrow X\}$, where $f$ is finite surjective flat morphism of constant degree power of $p$;

  \item cdh-covering.
 \end{enumerate}
\end{dfn*}
We call $\ell'$dh-topology what is called $\ell$dh-topology in \cite[\S5.2]{CD}.
Obviously, cdh-topology is coarser than $p$dh-topology, and $p$dh-topology is coarser than $\ell'$dh-topology for any $\ell\neq p$.

Let $S$ be an object of $\mr{Sch}_{/k}$.
Recall that the theorem of Temkin  \cite{Tem}, which is a refinement of Gabber's prime-to-$\ell$ alteration theorem, states as follows:
there exists an alteration $S'\rightarrow S$ whose {\em generic degree is some power of $p$} and $S'$ is smooth.
Without Temkin's theorem, $p$dh-topology might have been useless, but armed with the theorem, we can show the following statement as usual.

\begin{lem*}
 For any $S\in\mr{Sch}_{/k}$, there exists a $p$dh-covering $f\colon T\rightarrow S$ such that $T$ is a smooth $k$-scheme.
 We may even take $f$ to be proper.
\end{lem*}
\begin{proof}
 Even though the argument is standard, we recall a proof for the sake of completeness.
 We use the induction on the dimension of $S$.
 Using Temkin's theorem, take an alteration $T_1\rightarrow S$ whose generic degree is power to $p$ and $T_1$ is smooth.
 By using Gruson-Raynaud's flattening theorem, we may take a modification $S'\rightarrow S$ with center $Z\subset S$ such that the strict transform $T_2$ of $T_1$ is flat over $S'$.
 By construction $T_2\rightarrow S'$ is finite surjective flat morphism whose degree is power to $p$, and thus, $\{T_2\rightarrow S'\}$ is a $p$dh-covering.
 By induction hypothesis, we may find a proper $p$dh-covering $W\rightarrow Z$ such that $W$ is smooth.
 Because $\{Z,S'\rightarrow S\}$ is a $p$dh-covering, $\{W,T_2\rightarrow S\}$ is also a $p$dh-covering.
 This covering factors through $\{W,T_1\rightarrow S\}$, so the latter is a $p$dh-covering as well.
 Thus, we may simply take $T:=W\coprod T_1$.
\end{proof}
For any $S\in\mr{Sch}_{/k}$, we may find a $p$dh-hypercovering $S_{\bullet}\rightarrow S$ such that $S_i$ is $k$-smooth
by standard use of the lemma above and \cite[Expos\'e~$\mr{V}^{\mr{bis}}$, Proposition~5.1.3]{SGA}.


\subsection{}
\label{pdfdes}
We have the following $p$dh-descent, which is a straightforward corollary of a $\ell'$dh-descent result by S.~Kelly.

\begin{lem*}
 {\em Assume $p^{-1}\in\Lambda$}.
 Then any object of $\mc{D}_S$ satisfies $p$dh-descent.
 In other words, if we are given a $p$dh-hypercovering $p_\bullet\colon S_\bullet\rightarrow S$ and $\mc{F}\in\mc{D}_S$,
 the canonical morphism $\mc{F}\rightarrow\invlim_{i\in\mbf{\Delta}}p_{i*}p_i^*\mc{F}$ in the $\infty$-category $\mc{D}_S$ is an equivalence.
\end{lem*}
\begin{proof}
 Let $\mc{C}:=\mr{cofib}\bigl(\mc{F}\rightarrow\invlim_{i\in\mbf{\Delta}}p_{i*}p_i^*\mc{F}\bigr)$.
 We wish to show that $\mc{C}\simeq0$, and for this, it suffices to show that $\mc{C}\otimes_{\mb{Z}[1/p]}\mb{Z}_{(\ell)}\simeq0$
 for any prime $\ell\neq p$ (\emph{cf.}~\cite[proof of Proposition~3.13]{CD}).
 To show this, we must show that for any {\em compact} object $\mc{G}\in\mc{D}_S$, we have $\mr{Hom}(\mc{G},\mc{C}\otimes\mb{Z}_{(\ell)})\simeq0$.
 We have
 \begin{align*}
 \mr{Hom}(\mc{G},\mc{C}\otimes\mb{Z}_{(\ell)})
 \simeq
 \mr{cofib}\bigl[\mr{Hom}(\mc{G},\mc{F}\otimes\mb{Z}_{(\ell)})\rightarrow
 \mr{Hom}(\mc{G},(\invlim_{i\in\mbf{\Delta}}p_{i*}p_i^*\mc{F})\otimes\mb{Z}_{(\ell)})\bigr].
 \end{align*}
 We may further compute as
 \begin{align*}
 \mr{Hom}\bigl(\mc{G},(\invlim_{i\in\mbf{\Delta}}p_{i*}p_i^*\mc{F})\otimes\mb{Z}_{(\ell)}\bigr)
 &\simeq
 \mr{Hom}\left(\mc{G},\invlim_{i\in\mbf{\Delta}}p_{i*}p_i^*\mc{F}\right)\otimes\mb{Z}_{(\ell)}\\
 &\simeq
 \left(\indlim_{i\in\mbf{\Delta}}(\mr{Hom}(\mc{G},p_{i*}p_i^*\mc{F})\right)\otimes\mb{Z}_{(\ell)}\\
 &\simeq
 \indlim_{i\in\mbf{\Delta}}\mr{Hom}\left(p_i^*\mc{G},p_i^*\mc{F}\right)\otimes\mb{Z}_{(\ell)}\\
 &\simeq
 \indlim_{i\in\mbf{\Delta}}\mr{Hom}\left(p_i^*\mc{G},p_i^*\left(\mc{F}\otimes\mb{Z}_{(\ell)}\right)\right)\\
& \simeq
 \mr{Hom}\left(\mc{G},\invlim_{i\in\mbf{\Delta}}p_{i*}p_i^*\left(\mc{F}\otimes\mb{Z}_{(\ell)}\right)\right),
 \end{align*}
 where the $1^\text{st}$ and $4^\text{th}$ equivalences follow from the compactness of $\mc{G}$ and $p_i^*\mc{G}$ respectively.
 By \cite[Theorem~5.10]{CD}, $\mc{F}\otimes\mb{Z}_{(\ell)}$ admits $\ell'$dh-descent, in particular, $p$dh-descent.
 Thus, combining with the computations above, we have $\mc{C}\otimes\mb{Z}_{(\ell)}\simeq0$ as desired.
\end{proof}

Now, let $\mc{G}\in\mc{D}_S$. Then we have
\begin{align*}
 \Hom(\mc{G},\mc{F})\xrightarrow{\ \sim\ }
 \Hom\bigl(\mc{G},\invlim_{i\in\mbf{\Delta}}p_{i*}p_i^*\mc{F}\bigr)
 \cong
 \invlim_{i\in\mbf{\Delta}}\Hom\bigl(\mc{G},p_{i*}p_i^*\mc{F}\bigr)
 \cong
 \invlim_{i\in\mbf{\Delta}}\Hom\bigl(p_i^*\mc{G},p_i^*\mc{F}\bigr).
\end{align*}
We write $\mr{Hom}^k:=\pi_{-k}\mr{Hom}$.
Assume that $\Hom^k\bigl(p_i^*\mc{G}_i,p_i^*\mc{F}_i\bigr)\cong0$ for any $i\in\mbf{\Delta}$.
Then the complex of $\mbf{\Delta}$-indexed diagrams
$\bigl\{\Hom\bigl(p_i^*\mc{G}_i,p_i^*\mc{F}_i\bigr)\bigr\}_{i\in\mbf{\Delta}}$ belongs to $D^+(\mr{Ab}^{\mbf{\Delta}})$, and induces a spectral sequence
\begin{equation}
 \label{desSS}
  E_2^{p,q}=\mr{R}^p\invlim_{i\in\mbf{\Delta}}
  \mr{Hom}^q(p_i^*\mc{G},p_i^*\mc{F})\Longrightarrow
  \mr{Hom}^{p+q}(\mc{G},\mc{F}).
\end{equation}

\section{Main result and vanishing of higher homotopy}
\label{sec2}

\subsection{}
\label{FMbivdfn}
Let us recall the definition of bivariant theory after Fulton and MacPherson very briefly.

\begin{dfn*}
 A {\em bivariant theory $T$ over $k$} is an assignment to each morphism $f\colon X\rightarrow Y$
 in $\mr{Sch}_{/k}$ a $\mb{Z}$-graded Abelian group $T(f)$ equipped with three operations:
 \begin{enumerate}
  \item\label{def-1}(Product) For composable morphisms $f\colon X\rightarrow Y$ and $g\colon Y\rightarrow Z$,
       we have a homomorphism of graded groups $\bullet\colon T(f)\otimes T(g)\rightarrow T(g\circ f)$.
       
  \item (Pushforward) Assume we are given composable morphisms $f$ and $g$ as in~\eqref{def-1}. If, furthermore, $f$ is proper, we have the homomorphism $f_*\colon T(g\circ f) \rightarrow T(g)$.
	
  \item (Pullback) Consider the following Cartesian diagram:
	\begin{equation}
	 \label{stdcartdia}
	  \begin{gathered}\xymatrix{
	  X'\ar[r]^-{g'}\ar[d]_{f'}\ar@{}[rd]|\square&
	  X\ar[d]^{f}\\
	 Y'\ar[r]^-{g}&Y.
	  }\end{gathered}
	\end{equation}
	Then we have the homomorphism $g^*\colon T(f)\rightarrow T(f')$.
 \end{enumerate}
 These operations are subject to (more or less straightforward) compatibility conditions.
 Among these compatibility conditions, let us recall the projection formula for the later use.
 We consider the diagram \eqref{stdcartdia} such that $g$ is proper, and a morphism $h\colon Y\rightarrow Z$.
 Assume we are given $\alpha\in T(f)$ and $\beta\in T(h\circ g)$.
 Then we have $\alpha\bullet g_*(\beta)=g'_*(g^*\alpha\bullet\beta)$.

 Given bivariant theories $T$, $T'$, a morphism of theories $T\rightarrow T'$ is a collection of homomorphisms
 $T(f)\rightarrow T'(f)$ for any morphism $f$ in $\mr{Sch}_{/k}$ compatible with the operations above.
 We refer to \cite[\S2.2]{FM} for details.\footnote{
 In our situation, ``confined maps'' are ``proper morphisms'' and any
 Cartesian squares are ``independent squares''.}
\end{dfn*}

\begin{dfn}
 Let $T$ be a bivariant theory over $k$.
 An {\em $\mb{A}^1$-orientation of $T$} is an element $\eta\in T^1(\mb{A}^1\rightarrow\mr{Spec}(k))$, where $T^1$ is the degree $1$ part.
 Let $T'$ be another bivariant theory endowed with an $\mb{A}^1$-orientation $\eta'$.
 A morphism of bivariant theories $F\colon T\rightarrow T'$ is said to be {\em compatible with the orientation}
 if $F(\mb{A}^1\rightarrow\mr{Spec}(k))(\eta)=\eta'$.
\end{dfn}

\begin{rem*}
 Fulton and MacPherson called an {\em orientation} a rule to assign an element of $T(f)$ to each $f$ in a compatible manner.
 Since our $\mb{A}^1$-orientation can be regarded as a part of this data, we named it after Fulton and MacPherson's.
 This has {\it a priori} nothing to do with orientation of motivic spectra.
\end{rem*}

\subsection{}
Our Borel-Moore homology $\rHbm_a(X/S,\Lambda(b))$ defines a bivariant theory (in an extended sense because it is {\em bigraded}), \emph{cf.}~\cite[\S1.2.8]{Deg}.
By associating the graded group $\bigoplus_k\rHbm_{2k}(X/S,\Lambda(k))$ to $X\rightarrow S$,
we define the bivariant theory denoted by $\rHbm_{2*}(X/S,\Lambda(*))$.
This bivariant theory has a canonical orientation as follows.
Let $q\colon\mb{A}^1_S\rightarrow S$ be the projection.
Then we have a morphism
\begin{equation*}
 q_!q^*\Lambda_S(1)[2]
  \cong
  q_!q^!\Lambda_S
  \xrightarrow{\ \mr{adj}\ }
  \Lambda_S,
\end{equation*}
where the isomorphism is defined using \cite[Theorem~2.4.50.3]{CDbook} and the canonical identification of $\mr{MTh}_{\mb{A}^1}(T_q)$ with $\Lambda_{\mb{A}^1}(1)[2]$,
where $\mr{MTh}$ is the motivic Thom spectrum defined in \cite[Definition~2.4.12]{CDbook}.
The class of the above morphism in $\rHbm_2(X/S,\Lambda(1))\cong\pi_0\mr{Hom}_{D(S)}(q_!q^*\Lambda_S(1)[2],\Lambda_S)$
is the canonical orientation of $\rHbm_{2*}(X/S,\Lambda(*))$.

\subsection{}
Let us introduce another main player of this paper, $z(-,-)$, from \cite{SV}.
Let $f\colon X\rightarrow S$ be a morphism, and $d\geq0$ be an integer.
Recall that Suslin and Voevodsky\footnote{In fact, Suslin and Voevodsky used the
notation $z(X/S,d)$ as a presheaf on $\mr{Sch}_{/S}$. Our $z(X/S,d)$ is
the global sections of it.}
introduced Abelian groups $z_{\mr{equi}}(f,d)$ and $z(f,d)$, or $z_{\mr{equi}}(X/S,d)$ and $z(X/S,d)$ if no confusion may arise.
We do not recall the precise definition of these groups, but content ourselves with giving ideas of how these groups are defined.
Both groups are certain subgroups of the free Abelian group $Z(X)$ generated by integral subscheme of $X$.
If we are given an element $w\in Z(X)$ we may consider the ``support'' denoted by $\mr{Supp}(w)$ in an obvious manner.
Naively thinking, we wish to define $z(X/S,d)$ as a subgroup of $Z(X)$ consisting of $w$ such that $\mr{Supp}(w)\rightarrow S$
is equidimensional of dimension $d$ {\em over generic points of $S$}.
However, if we defined $z(X/S,d)$ in this way, the association $z(X_T/T,d)$ to $T$ would not be functorial.
In order to achieve this functoriality, Suslin and Voevodsky introduces an ingenious compatibility conditions.
We do not recall these compatibility conditions, but here is an illuminating example:
Let $Z\subset X$ be a closed immersion such that the morphism $Z\rightarrow S$ is flat.
Then the associated cycle $[Z]$, called a {\em flat cycle}, belongs to $z(X/S,d)$. 
%
%
%
Now, the group $z_{\mr{equi}}(X/S,d)$ is a subgroup of $z(X/S,d)$.
The element $w$ belongs to $z_{\mr{equi}}(X/S,d)$ if and only if the morphism $\mr{Supp}(w)\rightarrow S$ is equidimensional (of relative dimension $d$).
By the compatibility conditions we mentioned above, if we are given a morphism $S'\rightarrow S$,
we have the pullback homomorphism $z_{(\mr{equi})}(X/S,d)\rightarrow z_{(\mr{equi})}(X\times_S S'/S',d)$.
This enables us to define presheaves $\ul{z}_{(\mr{equi})}(X/S,d)$ on $\mr{Sch}_{/S}$.
Then $\ul{z}(X/S,d)$ is a cdh-sheaf, and the cdh-sheafification of $\ul{z}_{\mr{equi}}(X/S,d)$ coincides with $\ul{z}(X/S,d)$.
Furthermore, flat cycles generate $\ul{z}(X/S,d)$ cdh-locally, and can be thought of as a building pieces (\emph{cf.}~\cite[Theorem 4.2.11]{SV}).
The following theorem compactly summarizes some aspects of \cite{SV}.

\begin{thm*}[\emph{cf.} \protect{\cite{SV}}]
 The assignments $z(f,*):=\bigoplus_k z(f,k)$ and $z_{\mr{equi}}(f,*):=\bigoplus_k z_{\mr{equi}}(f,k)$ to a morphism $f$
 can be promoted to a bivariant theories with $\mb{A}^1$-orientation.
\end{thm*}
\begin{proof}
 Given any morphism $\alpha\colon T\rightarrow S$, the pullback homomorphism
 \[\alpha\colon z_{(\mr{equi})}(X/S,d)\longrightarrow z_{(\mr{equi})}(X_T/T,d)\]
 is then defined in \cite[right after Lemma 3.3.9]{SV}.
 Given a proper morphism $X\rightarrow Y$, the pushforward homomorphism
\[\beta_*\colon z_{(\mr{equi})}(X/S,d)\longrightarrow z_{(\mr{equi})}(Y/S,d)\] is defined in \cite[Corollary 3.6.3]{SV}.
Given a sequence of morphisms $X\xrightarrow{f}Y\xrightarrow{g}Z$ and integers $d,e\geq0$,
 the homomorphism
 \[\mr{Cor}\colon z_{(\mr{equi})}(X/Y,d)\times z_{(\mr{equi})}(Y/Z,e)\longrightarrow z_{(\mr{equi})}(X/Z,d+e)\]
 is defined in \cite[Corollary 3.7.5]{SV}.
 We may endow with $\mb{A}^1$-orientation by taking $\eta:=[\mb{A}^1]$.
 The compatibility conditions for these operations have also been proven in \cite{SV}.
\end{proof}

\subsection{}
\label{mainres}
Our main theorem is as follows.

\begin{thm*}
 Recall that the base field $k$ is a perfect field of characteristic $p>0$, and let $\Lambda:=\mb{Z}[1/p]$.
 Then, there exists a {\em unique} map of bivariant theories compatible with $\mb{A}^1$-orientation:
 \begin{equation*}
  \tau\colon z_{\mr{equi}}(-,*)\longrightarrow\rHbm_{2*}(-,\Lambda(*)).
 \end{equation*}
\end{thm*}

A proof of this theorem is given at the end of Section~\ref{pfmthm}.
Let us introduce a notation.
Let $f\colon X\rightarrow S$ be a flat morphism of relative dimension $d$.
Then $[X]$ is an element of $z_{\mr{equi}}(f,d)$.
If we are given $\tau$ as above, we have $\tau([X])\in\rHbm_{2d}(f,\Lambda(d))$.
This element is denoted by $\mr{Tr}^{\tau}_f$.

\begin{rem}\label{mainrem}
\leavevmode
 \begin{enumerate}
  \item Our theorem produces trace maps only for motivic Eilenberg-MacLane spectrum, and the reader might think that our theorem is too restrictive.
	However, this is not the case since the motivic Eilenberg-MacLane spectrum is universal among
	``absolute $\SH$-spectrum $\mb{E}$ with orientation which is $\Lambda$-linear and whose associated formal group law is additive'' by \cite[Remark 2.2.15]{Deg}.
	More precisely, if we are given such an absolute $\SH$-spectrum $\mb{E}$,
	we have a unique map $\phi\colon\mb{H}\Lambda\rightarrow\mb{E}$.
	Associated to this map, we may consider the composition
	$z_{\mr{equi}}(-,\Lambda(*))\xrightarrow{\tau}\rHbm_{2*}(-,\Lambda(*))\xrightarrow{\phi}\rHbm_{2*}(-,\mb{E}(*))$,
	where the last object is the Borel-Moore homology associated with $\mb{E}$,
	and we get trace maps for $\mb{E}$.

  \item Choose $\mb{E}$ to be the $\ell$-adic \'{e}tale absolute spectrum $\mb{H}_{\mr{\acute{e}t}}\mb{Q}_\ell$ for $\ell\neq p$.
	By construction above, we have $z_{\mr{equi}}(X/S,d)\rightarrow\rHbm_{\mr{\acute{e}t},2d}(X/S,d)$,
	where $\rHbm_{\mr{\acute{e}t},*}(X/S,*)$ is the $\ell$-adic Borel-Moore homology.
	If $f$ is a flat morphism of dimension $d$,
	the image of $[X]\in z_{\mr{equi}}(X/S,d)$ by this morphism is denoted by $\mr{Tr}^{\mr{\acute{e}t}}_f$.
	This element of $\rHbm_{\mr{\acute{e}t},2d}(X/S,d)$, considered as a morphism $\Hc(X/S,\Lambda_S(d)[2d])\rightarrow\Lambda_S$,
	coincides with the trace map defined in \cite[Expos\'e~XVIII, Th\'or\`eme~2.9]{SGA}.
	Thus, the morphism $\tau$ can be seen as a generalization of the trace map of \emph{loc.~cit.}, at least when the base field is perfect.
	
  \item When $X\rightarrow S$ is a g.c.i.~morphism, D\'{e}glise defined a similar map in \cite[Theorem~1]{Deg}.
	In fact, our map can be considered as a generalization of \cite{Deg} (even though we only consider over a field),
	or rather, is built upon D\'{e}glise's map.

  \item The theorem also holds in the case where $p=0$ and $\Lambda=\mb{Z}$.
	Furthermore, in the case where $p>0$ and if we assume the existence of the resolution of singularities,
	we may, in fact, take $\Lambda=\mb{Z}$ in the theorem.
	The proof works with obvious changes, and the detail is left to the reader.

  \item\label{mainrem-4}
       The theorem, in fact, holds for any field $k$, not necessarily perfect.
       In fact, let $l:=k^{\mr{perf}}$ be the perfection.
       The compact support cohomology $\Hc(X/S)$ is compatible with arbitrary base change.
       Thus, by \cite[Corollary 2.1.5]{EK}, or alternatively \cite[Proposition~8.1]{CD}, the pullback homomorphism
       $\rHbm_p(X/S,\Lambda(q))\rightarrow\rHbm_p(X_l/S_l,\Lambda(q))$ is an isomorphism since $p^{-1}\in\Lambda$.
       Using this isomorphism, the trace map for $\rHbm_{2*}(X_l/S_l,\Lambda(*))$, constructed above,
       induces the trace map for $\rHbm_{2*}(X/S,\Lambda(*))$ as well.
 \end{enumerate}
\end{rem}

\subsection{}
Before going to the next section, let us show the most important property to construct the trace map,
namely the vanishing of suitable higher homotopies.
For a morphism $f\colon X\rightarrow S$, we put $\dim(f):=\max\bigl\{\dim(f^{-1}(s))\mid s\in S\bigr\}$.

\begin{prop*}
 \label{vanishprop}
 For a morphism $f\colon X\rightarrow S$ in $\mr{Sch}_{/k}$ and an integer $d$ such that $\dim(f)\leq d$, we have
 \begin{equation*}
  \rHbm_{2m+n}(X/S,\Lambda(m))=0
 \end{equation*}
 in one of the following cases:
\begin{enumerate}
\item for any $m>d$ and any $n$,
\item when $m=d$ and for any $n>0$.
\end{enumerate}
\end{prop*}
\begin{proof}
 First, assume that $S=\mr{Spec}(k)$.
 We claim that
 \begin{equation*}
  \rHbm_n(X,m):=\pi_n\rHbm(X/\mr{Spec}(k),\Lambda(m))=0
 \end{equation*}
 if $m>d=\dim(X)$ or $m=d$ and
 $n>2m$. Assume $X$ is smooth of equidimension $d$.
 Then we know that
 $\mr{H}^{\mr{BM}}_n(X,\Lambda(m))\cong\mr{H}_{\mc{M}}^{2d-n}(X,\Lambda(d-m))\cong\mr{CH}^{d-m}(X,n-2m;\Lambda)$,
 where $\mr{H}_{\mc{M}}$ is the motivic cohomology and the last isomorphism follows by \cite[Example 11.2.3]{CDbook}.
 Thus the claim follows\footnote{
 In fact, this holds also for $n<0$ by
 \cite[III 2.5, II Ex.\ 1.16 (a)]{H}.
 }
 because $\mr{CH}^{0}(X,i;\Lambda)\cong\mr{H}_{\mc{M}}^{-i}(X,\Lambda(0))\cong\mr{H}_{\mr{zar}}^{-i}(X,\Lambda)=0$ for $i>0$.
 In general, we proceed by the induction on the dimension of $X$.
 We may assume $X$ is reduced.
 There exists $Z\subset X$ such that $X\setminus Z$ is smooth and
 $\dim(Z)<d$ since $k$ is assumed perfect.
 We have the exact sequence
 \begin{equation*}
  \cdots
  \longrightarrow
  \mr{H}^{\mr{BM}}_n(Z,\Lambda(m))
   \longrightarrow
   \mr{H}^{\mr{BM}}_n(X,\Lambda(m))
   \longrightarrow
   \mr{H}^{\mr{BM}}_n(X\setminus Z,\Lambda(m))
   \longrightarrow
   \mr{H}^{\mr{BM}}_{n-1}(Z,\Lambda(m))\longrightarrow\cdots.
 \end{equation*}
 Assume $d\leq m$. Then $\mr{H}^{\mr{BM}}_n(Z,\Lambda(m))=0$ for any $n$
 since $\dim(Z)<d\leq m$ and the induction hypothesis. Thus
 $\mr{H}^{\mr{BM}}_n(X,\Lambda(m))\cong\mr{H}^{\mr{BM}}_n(U,\Lambda(m))$,
 and the claim follows by the smooth case we have already treated.
 We next assume that $S$ is smooth over $k$. We may assume that $S$ is
 of equidimension $e$. Let $\pi\colon S\rightarrow\mr{Spec}(k)$ be
 the structural morphism. Then we have
 \begin{align*}
  \mr{Hom}\bigl(\Hc(X/S,\Lambda_S)\!\left<m\right>,\Lambda_S[-n]
  \bigr)
  &\cong
  \Hom\bigl(\Hc(X/S,\Lambda_S)\!\left<m\right>,\pi^*\Lambda[-n]
  \bigr)\\
  &\cong
  \Hom\bigl(\Hc(X/S,\Lambda_S)\!\left<m\right>,\pi^!\Lambda
  \left<-e\right>[-n]
  \bigr)\\
  &\cong
  \Hom\bigl(\Hc(X)\!\left<m+e\right>,\Lambda[-n]\bigr).
 \end{align*}
 Since $\dim(X)\leq\dim(S)+d=e+d$, we get the vanishing by the $S=\mr{Spec}(k)$ case.

 Finally, we treat the general case.
 We take a $p$dh-hypercovering $S_\bullet\rightarrow S$ so that $S_i$ is smooth.
 Let $\mc{F},\mc{G}\in D(S)$.
 Then by $p$dh-descent spectral sequence \eqref{desSS}, we have
 \begin{equation*}
  E_2^{p,q}=\mr{R}^p\invlim_{\mbf{\Delta}}
   \mr{Hom}^q(\mc{F}_\bullet,\mc{G}_\bullet)\Longrightarrow
  \mr{Hom}^{p+q}(\mc{F},\mc{G}).
 \end{equation*}
 If $E_2^{p,q}=0$ for $q<0$, then $\mr{Hom}^i(\mc{G},\mc{F})=0$ for
 $i<0$. Thus, we get the claim by applying this to
 $\mc{F}=\Hc(X/S,\Lambda_S)\!\left<m\right>$ and
 $\mc{G}=\Lambda$.
\end{proof}

\begin{rem*}
 Consider the case where $p$ may not be invertible in $\Lambda$.
 If $S$ is smooth, then the proposition holds.
 If we further assume the resolution of singularities, the proposition also holds for any $f$.
\end{rem*}

\section{Construction of the trace map when the base is smooth}
\label{sec3}
Let $f\colon X\rightarrow S$ be a flat morphism. When $S$ is smooth,
we will construct a map which is supposed to be the same as
$\mr{Tr}_f^{\tau}$ in this section.

\subsection{}
\label{constmapsmo}
For a scheme $Z$, we often denote $\dim(Z)$ by $d_Z$.
Let $f\colon X\rightarrow S$ be (any) separated morphism of finite type such that $S$ is smooth equidimensional, and put $d_f:=d_X-d_S$.
In this case, let us construct a morphism $t_f\colon f_!\Lambda_X\!\left<d_f\right>\rightarrow\Lambda_S$,
which we will show to be equal to $\mr{Tr}_f$ when $f$ is flat.

Let us start to construct $t_f$. Considering componentwise, it suffices to construct the morphism when $S$ is connected.
For any separated scheme $X$ of finite type over $k$, we have the canonical isomorphism
\begin{equation*}
 \gamma_X\colon\rHbm_{2n}(X,\Lambda(n))
  \xrightarrow{\ \sim\ }
  \mr{CH}_n(X;\Lambda)
\end{equation*}
by \cite[Corollary~3.9]{J}. We have
\begin{align*}
 \rHbm_{2d_f}(X/S,\Lambda(d_f))
 \cong
 \rHbm_{2d_X}(X,\Lambda(d_X))
 \xrightarrow[\sim]{\ \gamma_X\ }
 \mr{CH}_{d_X}(X;\Lambda),
\end{align*}
where the first isomorphism follows since $g^*\!\left<d\right>\xrightarrow{\sim}g^!$ for any equidimensional smooth morphism $g$ of relative dimension $d$.
Let $X=\bigcup_{i\in I} X_i$ be the irreducible components, and let $I'\subset I$ be the subset of $i$ such that $\dim(X_i)=d_X$.
Let $\xi_i$ be the generic point of $X_i$.
The element in $\rHbm_{2d}(X/S,\Lambda(d))$ corresponding via the isomorphism above to the element
$\sum_{i\in I'}\mr{lg}(\mc{O}_{X,\xi_i})\cdot[X_{i,\mr{red}}]\in\mr{CH}_{d_X}(X;\Lambda)$ on the right hand side is defined to be $t_f$.

Let us end this paragraph with a simple observation.
Let $U\subset X$ be an open dense subscheme.
Then the restriction map $\rHbm_{2d_f}(X/S,\Lambda(d_f))\rightarrow\rHbm_{2d_f}(U/S,\Lambda(d_f))$ is an isomorphism.
Indeed, we have $\rHbm_{2d_h}(X/Z,\Lambda(d_h))\cong\mr{CH}_{d_X}(X;\Lambda)\cong\Lambda^{\oplus r_X}$,
where $r_X$ is the set of irreducible components of $X$ of dimension $d_X$ by the computation above.
Since $r_X$ and $r_U$ are the same, we get the claim.

\subsection{}
By the setup \ref{fixcoh}, we may apply \cite[Introduction, Theorem~1]{Deg}.
In particular, for a morphism between smooth schemes $f\colon X\rightarrow Y$ we have the fundamental class
$\overline{\eta}_f\in\rHbm_{2d_f}(X/Y,\Lambda(d_f))$.
When $Y=\mr{Spec}(k)$, we sometimes denote $\overline{\eta}_f$ by $\overline{\eta}_X$.
As we expect, we have the following comparison.
\begin{lem*}\label{twodefcoin}
 Assume $f\colon X\rightarrow Y$ is a morphism between smooth equidimensional schemes.
 Then $t_f=\overline{\eta}_f$ in $\rHbm_{2d_f}(X/Y,\Lambda(d_f))$.
\end{lem*}
\begin{proof}
 Assume $Y=\mr{Spec}(k)$.
 In this case, $f$ is smooth.
 Then by \cite[Theorem~2.5.3]{Deg}, the fundamental class $\overline{\eta}_f$ is equal to the one constructed in \cite[Proposition~2.3.11]{Deg},
 which is nothing but the one we constructed above by \cite[Proposition~3.12]{J}.
 Let us treat the general case. For a $k$-scheme $Z$, denote by $p_Z$
 the structural morphism.
 Unwinding the definition, our $t_f$ is the unique dotted map so that the following diagram on the right is commutative:
 \begin{equation*}
 \xymatrix{
  \Lambda_X\!\left<d_f\right>\ar@{.>}[r]
  \ar[rd]_{t^{\mr{adj}}_{p_X}}&
  f^!\Lambda_Y
  \ar[d]^-{f^!\overline{\eta}_Y} & &f_!\Lambda_X\!\left<d_f\right>\ar@{.>}[r]
  \ar[d]_{f_!t^{\mr{adj}}_{p_X}}&
  \Lambda_Y
  \ar[d]^-{\overline{\eta}_Y}\\
  &f^!p_Y^!\Lambda\!\left<-d_Y\right>, & &
   f_!p_X^!\Lambda_X\!\left<-d_Y\right>
  \ar[r]_-{\mr{adj}_f}&
  p_Y^!\Lambda\!\left<-d_Y\right>.
  }
 \end{equation*}
 Here, $t^{\mr{adj}}_g$ denotes the morphism given by taking adjoint to $t_g$.
 Equivalently, $t_f$ is the unique dotted map so that the diagram above on the left is commutative.
 Thus, it suffices to check that the diagram replacing the
 dotted arrow by $\overline{\eta}_f$ commutes. From what we have
 checked, $t^{\mr{adj}}_{p_Z}=\overline{\eta}_{p_Z}^{\mr{adj}}$ for any
 smooth scheme $Z$. Thus, the desired commutativity follows by the
 associativity property of fundamental class
 (\emph{cf.}~\cite[Introduction, Theorem~1.2]{Deg}).
\end{proof}

\begin{lem}\label{bccommsim}
\leavevmode
\begin{enumerate}
  \item\label{bccommsim-1}
       Assume we are given morphisms $X\xrightarrow{f}Y\xrightarrow{g}Z$ such that $Y$ and $Z$ are smooth and equidimensional. 
       Let the composition be $h$.
       We have $t_g\bullet t_f=t_h$ in $\rHbm_{2d_h}(X/Z,\Lambda(d_h))$.
	
  \item\label{bccommsim-2}
       Consider the Cartesian diagram \eqref{stdcartdia}.
       Assume further that $Y$ and $Y'$ are smooth equidimensional and $f$ is flat.
       The map $g^*\colon\rHbm_{2d_f}(X/Y,\Lambda(d_f))\rightarrow\rHbm_{2d_{f}}(X'/Y',\Lambda(d_{f}))$ sends $t_f$ to $t_{f'}$.
       
  \item\label{bccommsim-3}
       Consider a proper morphism $f\colon X\rightarrow Y$ and a morphism $g\colon Y\rightarrow Z$ such that $Z$ is smooth and equidimensional.
       Put $h:=g\circ f$.
       Then the map $f_*\colon\rHbm_{2d_h}(X/Z,\Lambda(d_h))\rightarrow\rHbm_{2d_h}(Y/Z,\Lambda(d_h))$
       sends $t_h$ to $\deg(X/Y)\cdot t_g$ when $t_h=t_g$ and $0$ otherwise.
 \end{enumerate} 
\end{lem}
\begin{proof}
 Let us check the claim \eqref{bccommsim-1} of the lemma. By construction of $t_f$, we may assume
 that $X$ is reduced.
 By \S\ref{constmapsmo}, we may shrink $X$ by its dense open subscheme since $\rHbm_{2d_h}(X/Z,\Lambda(d_h))$ remains the same.
 Thus, we may assume that $X$ is smooth as well.
 In this case, we get the compatibility by Lemma \ref{twodefcoin} and
 \cite[Introduction, Theorem~1.2]{Deg}.
 The final claim \eqref{bccommsim-3} is just a reformulation of \cite[Proposition~3.11]{J}.
 Let us check the claim \eqref{bccommsim-2} of the lemma.
 Since $Y$, $Y'$ are smooth, we may factor $g$ into a regular immersion followed by a smooth morphism.
 Thus, it suffices to check the case where $g$ is a regular immersion and a smooth morphism separately.
 In both cases, consider the following diagram:
 \begin{equation*}
  \xymatrix@C=30pt{
   f'_!\Lambda_{X'}\!\left<d_f\right>
   \ar@{-}[d]_{\sim}\ar@/^15pt/[drrr]^(.7){t^{\mr{adj}}_{p_{X'}}}
   \ar@{.>}@/_25pt/@<-4ex>[dd]\ar@{}[drr]|{\clubsuit}&&&\\
   g^*f_!\Lambda_X\!\left<d_f\right>
   \ar[r]_-{t^{\mr{adj}}_{p_X}}\ar@{.>}[d]&
   g^*f_!p_X^!\Lambda\!\left<-d_Y\right>
   \ar@{-}[r]^-{\sim}\ar[d]^{\mr{adj}_f}&
   f'_!g'^*p_X^!\Lambda\!\left<-d_Y\right>
   \ar[r]_-{\overline{\eta}^{\mr{adj}}_{g'}}&
   f'_!g'^!p_X^!\Lambda\!\left<-d_{Y'}\right>
   \ar[d]^{\mr{adj}_{f'}}\\
  g^*\Lambda_Y\ar[r]^-{t^{\mr{adj}}_{p_Y}}
   \ar@/_20pt/[rrr]_(.3){t^{\mr{adj}}_{p_{Y'}}}&
   g^*p_Y^!\Lambda\!\left<-d_Y\right>
   \ar[rr]^-{\overline{\eta}^{\mr{adj}}_{g}}&
   &
   g^!p_Y^!\Lambda\!\left<-d_{Y'}\right>.
   }
 \end{equation*}
 The map $g^*t_f$ is the unique straight dotted arrow redering the left small square diagram commutes,
 and $t_{f'}$ is the unique bent dotted arrow rendering the outer largest diagram commutes.
 Since $f$ is flat, $f$ is transversal to $g$ in the sense of \cite[Example~3.1.2]{Deg}.
 This implies that $f^*(\overline{\eta}_{g})=\overline{\eta}_{g'}$ by \cite[Introduction, Theorem~1.3]{Deg}.
 By taking the adjoint, this implies that the right square is commutative.
 Since $Y$, $Y'$ are assumed to be smooth, we have $t^{\mr{adj}}_{p_Y}=\overline{\eta}^{\mr{adj}}_{p_Y}$ and
 $t^{\mr{adj}}_{p_{Y'}}=\overline{\eta}^{\mr{adj}}_{p_{Y'}}$ by the previous lemma.
 Since $g$, $p_Y$, $p_{Y'}$ are gci morphism, the bottom semicircular diagram is commutative by \cite[Introduction, Theorem~1.2]{Deg}.
 In order to check the equality in the claim, it remains to check that the $\clubsuit$-marked diagram commutes.

 When $g$ is smooth, the verification is easy, so we leave it to the reader.
 Assume $g$ is a regular immersion. In \cite[Definition~2.31]{J}, Jin defines a morphism
 $R_f(g)\colon\Hc(X')\rightarrow\Hc(X)\!\left<c\right>$ where $c=\dim(Y)-\dim(Y')$.
 By construction, this is defined as the composition
 \begin{equation*}
  p_{X'!}\Lambda_{X'}\cong p_{X!}g'_!f'^*\Lambda_{Y'}\cong
   p_{X!}f^*g_!\Lambda_{Y'}\xrightarrow{\ \overline{\eta}_g\ }
   p_{X!}f^*\Lambda_{Y}\left<c\right>.
 \end{equation*}
 Applying \cite[Introduction, Theorem~1.3]{Deg}, this is the same as $p_{X!}\overline{\eta}_{g'}$.
 Now, since $g^!([X])=[X']$ in $\mr{CH}_{d_{X'}}(X')$ by the flatness of $f$,
 \cite[Proposition~3.15]{J} implies that the following diagram on the left commutes:
 \begin{equation*}
  \xymatrix@C=25pt{
   \Hc(X')
   \ar[rd]_{t_{p_{X'}}}\ar[rr]^-{R_f(g)=p_{X!}\overline{\eta}_{g'}}&&
   \Hc(X)\!\left<c\right>\ar[ld]^-{t_{p_X}}\\
  &\Lambda\!\left<-d_{X'}\right>,&}
   \qquad
   \xymatrix{
   g'^*p_X^*\Lambda\!\left<d_{X'}\right>
   \ar[r]^-{\overline{\eta}_{g'}}\ar[d]_{t_{p_X}}&
   g'^!p_X^*\Lambda\!\left<d_X\right>
   \ar[d]^{t_{p_X}}\\
  g'^*p_X^!\Lambda\!\left<c\right>
   \ar[r]^-{\overline{\eta}_{g'}}&
   g'^!p_X^!\Lambda.
   }
 \end{equation*}
 Taking the adjunction, the verification is reduced to the commutativity of the right diagram above.
 This follows by the following commutative diagram:
 \begin{equation*}
   \xymatrix@C=40pt{
   g'^*p_X^*\Lambda\!\left<d_{X'}\right>
   \ar[d]_{g^*(t_{p_X})}\ar[r]^-{\overline{\eta}_{g'}\otimes\mr{id}}&
   g'^!\Lambda_X\otimes g'^*p_X^*\Lambda\!\left<d_X\right>
   \ar[r]^-{\mr{proj}}\ar[d]^{\mr{id}\otimes g'^*(t_{p_X})}&
   g'^!p_X^*\Lambda\!\left<d_X\right>
   \ar[d]^{g'^!(t_{p_X})}\\
  g'^*p_X^!\Lambda\!\left<c\right>\ar[r]^-{\overline{\eta}_{g'}\otimes\mr{id}}&
   g'^!\Lambda_X\otimes g'^*p_X^!\Lambda
   \ar[r]^-{\mr{proj}}&
   g'^!p_X^!\Lambda.
   }  
 \end{equation*}
 Here, $\mr{proj}$ are the morphisms induced by the projection formula (or more precisely \cite[(1.2.8.a)]{Deg}),
 and we conclude the proof.
\end{proof}

\begin{lem}
 \label{smcomp}
 Assume we have a morphism of bivariant theories $\tau$ in Theorem {\normalfont\ref{mainres}}.
 Then for a flat morphism $f\colon X\rightarrow S$ such that $S$ is smooth and equidimensional,
 we must have an equality $\mr{Tr}^{\tau}_f=t_f$.
\end{lem}
\begin{proof}
 First, consider the case where $X=S$.
 Since $\tau$ preserves the product structure, $\tau(\mr{id}_S)$ must send the unit element
 $\mbf{1}=[S]\in z_{\mr{equi}}(S/S,0)$ to $\mbf{1}=\mr{id}\in\rHbm_0(S/S,0)$.
 By \cite[Proposition~3.12]{J}, $t_{\mr{id}}$ is equal to $\mr{id}$ as well, and the claim follows in this case.
 When $f$ is an open immersion, we may argue similarly.

 Now, let $f\colon X\rightarrow S$ be a finite \'{e}tale morphism such that $S$ is smooth and equidimensional of dimension $d$.
 We may assume $X$ and $S$ are integral, and the degree of $f$ is $n$.
 By $f_*\colon z_{\mr{equi}}(X/S,0)\rightarrow z_{\mr{equi}}(S/S,0)$,
 $[X]$ is sent to $n\cdot[S]$ in $z_{\mr{equi}}(S/S,0)$ by definition of $f_*$.
 This implies that $f_*(\mr{Tr}^{\tau}_f)=n\cdot\mr{id}$ where $f_*\colon\rHbm_0(X/S,\Lambda(0))\rightarrow\rHbm_0(S/S,\Lambda(0))$.
 On the other hand, we have the following commutative diagram by \cite[Proposition~3.11]{J}:
 \begin{equation*}
 \xymatrix@C=40pt{
  \Lambda\ar@{-}[r]^-{\sim}\ar[d]_{n\cdot}&
  \mr{CH}_{d}(X;\Lambda)\ar[r]^-{\gamma_X}_-{\sim}\ar[d]_{f_*}&
  \rHbm_{2d}(X,\Lambda(d))\ar[d]^{f_*}\ar@{-}[r]^-{\sim}&
  \rHbm_0(X/S,\Lambda(0))\ar[d]^{f_*}\\
  \Lambda\ar@{-}[r]^-{\sim}&
 \mr{CH}_{d}(S;\Lambda)\ar[r]^-{\gamma_S}_-{\sim}&
  \rHbm_{2d}(S,\Lambda(d))\ar@{-}[r]^-{\sim}&
  \rHbm_0(S/S,\Lambda(0)).
  }
 \end{equation*}
 This implies that, since $\Lambda=\mb{Z}[1/p]$ is torsion free, the left vertical map is injective, and so is the right vertical map.
 Thus $\mr{Tr}^{\tau}_f$ is characterized by the property that $f_*\mr{Tr}^{\tau}_f=n\cdot\mbf{1}$,
 and it suffices to check that $f_*t_f=n\cdot\mbf{1}$.
 By definition, $\gamma_{X}([X])=t_f$, and the commutative diagram again implies that $f_*t_f=n\cdot\mbf{1}$.
 Thus $t_f=\mr{Tr}^{\tau}_f$ in this case.

 Consider the case where $S=\mr{Spec}(k)$ .
 We may assume that $X$ is integral, and we may shrink $X$ by its open dense subscheme since $\rHbm_{2d_X}(X,\Lambda(d_X))$ does not change by \S\ref{constmapsmo}.
 Then we may assume that $f$ can be factored into $X\xrightarrow{g}\mb{A}^{d_f}\rightarrow\mr{Spec}(k)$ where the first morphism is \'{e}tale.
 By shrinking $X$ further, we may assume we have the factorization $X\xrightarrow{g'}V\hookrightarrow\mb{A}^{d_f}$ of $g$
 where $g'$ is finite \'{e}tale.
 Since the trace map is assumed to preserve $\mb{A}^1$-orientation, we must have $\mr{Tr}^{\tau}_p=t_p$
 where $p\colon\mb{A}^1\rightarrow\mr{Spec}(k)$ by \cite[Proposition~3.12]{J}.
 Thus, by Lemma \ref{bccommsim}-\eqref{bccommsim-1}, we have $\mr{Tr}^{\tau}_f=t_f$.

 Finally, let us treat the general case.
 Let $U\subset X$ be an open dense subscheme such that $U_{\mr{red}}$ is smooth over $k$.
 Let $e$ be the dimension of $S$.
 We have an isomorphism $F\colon\rHbm_{2d}(X/S,\Lambda(d))\simeq\rHbm_{2(d+e)}(U,\Lambda(d+e))$, again, by \S\ref{constmapsmo}.
 By construction, this morphism sends $x$ to $\overline{\eta}_{S}\bullet x$.
 In view of Lemma \ref{twodefcoin}, this is equal to $t_S\bullet x$.
 Now, we have
 \begin{align*}
  F(\mr{Tr}^{\tau}_f)=\mr{Tr}^{\tau}_{f|_U}\bullet t_S=
  \mr{Tr}^{\tau}_{f|_U}\bullet\mr{Tr}^{\tau}_S=\mr{Tr}^{\tau}_U
  =t_U=t_{f|_U}\bullet t_S=F(t_f)
 \end{align*}
 where the $2^\text{nd}$ equality follows by what we have already proven, the $3^\text{rd}$ by
 the transitivity of the trace map, the $4^\text{th}$ by what we have already proven,
 and the $5^\text{th}$ by Lemma~\ref{bccommsim}-\eqref{bccommsim-1}. Thus, we conclude the
 proof.
\end{proof}

\section{Construction of the trace map}
\label{pfmthm}
In this section, we prove the main result.

\subsection{}
Let $f\colon X\rightarrow S$ be a morphism.
To a morphism $T\rightarrow S$, we associate
\begin{equation*}
 \ulrHbm(X/S,n)(T):=\Hom_{\mc{D}_S}\bigl(\Hc(X/S,\Lambda_S(n)),\H^*(T/S,\Lambda_S)\bigr)
  \simeq\rHbm(X_T/T,n),
\end{equation*}
which defines a presheaf of spectra $\ulrHbm(X/S,n)$ on $\mr{Sch}_{/S}$.
We denote by $\ulrHbm_{m}(X/S,n)$ the Abelian presheaf $\pi_m\ulrHbm(X/S,n)$ on $\mr{Sch}_{/S}$.
Here, $\pi_m$ is taken as a presheaf and do not consider any topology.

\begin{lem}\label{sHbmprop}
\leavevmode
 \begin{enumerate}
  \item\label{sHbmprop-1}
       The spectra-valued presheaf $\ulrHbm(X/S,n)$ on $\mr{Sch}_{/S}$ is a spectra-valued sheaf on the $\infty$-topos $\Shv(\mr{Sch}_{/S,p\mr{dh}})^{\wedge}$,
       where $\mr{Sch}_{/S,p\mr{dh}}$ denotes the $p$dh-site and $(\cdot)^{\wedge}$ denotes the hypercompletion.
       
  \item\label{sHbmprop-2}
       Let us assume that $\dim(f)\leq d$. We the have $\Gamma\bigl(T,\widetilde{\pi}_{2d}\ulrHbm(X/S,d)\bigr)\cong\pi_{2d}\rHbm(X_T/T,d)$,
       and $\pi_{i}\ulrHbm(X/S,d)=0$ for $i>2d$.
       Here, $\widetilde{\pi}_n$ is the functor $\pi_n$ in the $\infty$-topos $\Shv(S_{p\mr{dh}})^{\wedge}$,
       in other words, the $p$dh-sheafification of $\pi_n$.
 \end{enumerate} 
\end{lem}
\begin{proof}
 Let us show the claim \eqref{sHbmprop-1} of the lemma.
 Let $T_\bullet\rightarrow T$ be a $p$dh-hypercovering of $q\colon T\rightarrow S\in\mr{Sch}_{/S}$.
 We must show that the canonical map
 \begin{equation*}
  \ulrHbm(X/S,n)(T)\rightarrow
   \invlim_{i\in\mbf{\Delta}}\ulrHbm(X/S,n)(T_i)
 \end{equation*}
 is an equivalence in $\mc{D}_T$.
 By Lemma \ref{pdfdes} applied to $\mc{F}=\Lambda_T\in\mc{D}_T$,
 we have the equivalence $\Lambda_T\xrightarrow{\sim}\invlim_{i\in\mbf{\Delta}}\H^*(T_i/T,\Lambda_T)$.
 By applying $q_*$, taking into account that $q_*$ commutes with arbitrary limit by the existence of a left adjoint, we have an equivalence
 $\H^*(T/S,\Lambda_S)\xrightarrow{\sim}\invlim_{i\in\mbf{\Delta}}\H^*(T_i/S,\Lambda_S)$.
 Thus, the claim follows by definition.
%
%
 Let us show the claim \eqref{sHbmprop-2} of the lemma.
 The Abelian sheaf $\widetilde{\pi}_i\ulrHbm(X/S,d)$ is the $p$dh-sheafification of the Abelian presheaf associating $\pi_i\ulrHbm(X/S,d)(T)$ to $T\in\mr{Sch}_{/S}$.
 Since $\pi_i\ulrHbm(X/S,d)(T)\cong\rHbm_i(X_T/T,d)$, this vanishes if $i>2d$ by Proposition \ref{vanishprop}.
 Furthermore, since $\invlim$ is left exact, \ref{sHbmprop-1} and the vanishing for $i>2d$ imply that $\pi_{2d}\ulrHbm(X/S,d)$ is already a $p$dh-sheaf on $\mr{Sch}_{/S}$,
 and the claim follows.
\end{proof}

\subsection{}
Let $X\rightarrow S$ be a morphism. Let us recall the Abelian group $\mr{Hilb}(X/S,r)$ for an integer $r\geq0$ from \cite[\S3.2]{SV}.
This is the set of closed subschemes in $X$ which are flat over $S$.
We denote by $\Lambda\mr{Hilb}(X/S,r)$ the free $\Lambda$-module generated by $\mr{Hilb}(X/S,r)$.

Now, assume that $S$ is smooth.
For a (flat) morphism $g\colon Z\rightarrow S$ in $\mr{Hilb}(X/S,r)$, we constructed $t_g\in\rHbm_{2r}(Z/S,\Lambda(r))$ in \S\ref{constmapsmo} when $S$ is equidimensional.
Even if $S$ is not equidimensional, by considering componentwise, we define the element $t_g$.
By associating to $Z$ the image of $t_g$ via the map $\rHbm_{2r}(Z/S,\Lambda(r))\rightarrow\rHbm_{2r}(X/S,\Lambda(r))$,
we have the map
$\mr{Hilb}(X/S,r)\rightarrow\rHbm_{2r}(X/S,\Lambda(r))$.
This yields the map
$\Lambda\mr{Hilb}(X/S,r)\rightarrow\rHbm_{2r}(X/S,\Lambda(r))$.
Now, let $I_{X/S}\subset\Lambda\mr{Hilb}(X/S,r)$ be the submodule
consisting of elements $\sum \lambda_iZ_i\in\Lambda\mr{Hilb}(X/S,r)$
such that the associated cycle $\sum \lambda_i[Z_i]=0$
(\emph{cf.}~the paragraph before Theorem~4.2.11 in~\cite{SV}).\footnote{
In \cite[\S2.1]{K}, Kelly pointed out a problem in the definition of the
map $\mr{cycl}$ of \cite{SV} used in the definition of $I_{X/S}$ above.
Note that we may employ Kelly's definition of $\mr{cycl}$ to define
$I_{X/S}$, but we get the same ideal, and it does not affect our
arguments.}
Since $t_g$ only depends on the underlying subset and its length, the above constructed map factors through $I$,
and defines a map
\begin{equation*}
 T(X/S,r)\colon
 \Lambda\mr{Hilb}(X/S,r)/I_{X/S}
  \rightarrow
  \rHbm_{2r}(X/S,\Lambda(r))
\end{equation*}

\begin{lem}
 \label{compat}
 Let $h\colon T\rightarrow S$ be a morphism between smooth
 $k$-schemes. Then we have the following commutative diagram of Abelian
 groups
 \begin{equation*}
  \xymatrix@C=50pt{
   \Lambda\mr{Hilb}(X/S,r)\ar[r]^-{T(X/S,r)}
   \ar[d]_{h^*}&
   \rHbm_{2r}(X/S,\Lambda(r))\ar[d]^{h^*}\\
  \Lambda\mr{Hilb}(X_T/T,r)\ar[r]^-{T(X_T/T,r)}&
   \rHbm_{2r}(X_T/T,\Lambda(r))
   }
 \end{equation*}
\end{lem}
\begin{proof}
 This follows immediately from Lemma \ref{bccommsim}-\eqref{bccommsim-2}.
\end{proof}

\subsection{}
Let $f\colon X\rightarrow S$ be a morphism.
Let $\mc{Z}(X/S,r)$ be the presheaf of Abelian groups on $\mr{Sch}_{/S}$ which sends $T$ to $\Lambda\mr{Hilb}(X_T/T,r)/I_{X_T/T}$,
and $\ul{z}(X/S,r)$ be the presheaf which sends $T$ to $z(X_T/T,r)$.
Consider the (geometric) morphism of sites $\mr{Sch}_{/S,p\mr{dh}}\xrightarrow{a}\mr{Sch}_{/S,\mr{cdh}}\xrightarrow{b}\mr{Sch}_{/S}$. Then we have
\begin{equation}
 \label{sheafifizH}
 (b\circ a)^*\bigl(\mc{Z}(X/S,r)\bigr)
  \cong
  a^*\bigl(b^*\mc{Z}(X/S,r)\bigr)
  \cong
  a^*\ul{z}(X/S,\Lambda(r))
  \cong
  \ul{z}(X/S,\Lambda(r)),
\end{equation}
where the $2^\text{nd}$ isomorphism follows by \cite[Theorem~4.2.11]{SV},
the last isomorphism follows since $z(X/S,\Lambda(r))$ is an h-sheaf by \cite[Theorem~4.2.2]{SV} and, in particular, a $p$dh-sheaf.

Now, a $p$dh-hypercovering $S_\bullet\rightarrow S$ is said to be {\em good} if $S_i$ is smooth for any $i$.
Let $\mr{HR}(S)$ be the (ordinary) category of $p$dh-hypercoverings of $S$ (\emph{cf.}~\cite[Expos\'e~V, \S7.3.1]{SGA}).
Denote by $\mr{HR}^{\mr{g}}(S)$ the full subcategory of $\mr{HR}(S)$ consisting of good $p$dh-covers.
Recall that $\mr{HR}(S)^{\mr{op}}$ is filtered (\emph{cf.}~\cite[Expos\'e~V, Th\'eor\`eme~7.3.2]{SGA}).
For any $S_\bullet\in\mr{HR}(S)$, we can take $S'_\bullet\in\mr{HR}^{\mr{g}}(S)$ and a morphism $S'_\bullet\rightarrow S_\bullet$ by
\cite[Expos\'e~$\mr{V}^{\mr{bis}}$, Proposition~5.1.3]{SGA} and \ref{GTres},
which implies that $\mr{HR}^{\mr{g}}(S)^{\mr{op}}$ is cofinal in $\mr{HR}(S)^{\mr{op}}$  (\emph{cf.}~\cite[Expos\'e~I, Proposition~8.1.3]{SGA}).
Put $X_\bullet:=X\times_S S_\bullet$.
Thus we have the isomorphisms
\begin{align*}
 z(X/S,\Lambda(r))
 \cong
 \ul{z}(X/S,\Lambda(r))(S)
 &\xrightarrow{\ \sim\ }
 \indlim_{S_\bullet\in\mr{HR}(S)^{\mr{op}}}
 \invlim_{i\in\mbf{\Delta}}\mc{Z}(X/S,r)(S_i)\\
 &\xleftarrow{\ \sim\ }
 \indlim_{S_\bullet\in\mr{HR}^{\mr{g}}(S)^{\mr{op}}}
 \invlim_{i\in\mbf{\Delta}}
 \Lambda\mr{Hilb}(X_i/S_i,r)/I,
\end{align*}
where the $2^\text{nd}$ isomorphism holds by \cite[Expos\'e~V, Th\'eor\`eme~7.4.1]{SGA} and \eqref{sheafifizH}.

\medskip
Let $\mr{Ab}_{\mbf{\Delta}}$ be the category of simplicial Abelian groups.
Consider the functors
\begin{equation*}
 \mc{Z}(-,r),\,\rHbm_{2r}(-,\Lambda(r))
  \colon
  \mr{HR}^{\mr{g}}(S)^{\mr{op}}\rightarrow
  \mr{Ab}_{\mbf{\Delta}},
\end{equation*}
defined by sending $S_{\bullet}$ to $\mc{Z}(X_\bullet/S_\bullet,r)$ and $\rHbm_{2r}(X_\bullet/S_\bullet,\Lambda(r))$ respectively.
By Lemma \ref{compat}, we have the map of functors $\mc{Z}(-,r)\rightarrow\rHbm_{2r}(-,\Lambda(r))$.
Now, assume:
\begin{quote}
 $\dim(f^{-1}(s))\leq d$ for any $s\in S$.
\end{quote}
By Lemma \ref{sHbmprop}-\eqref{sHbmprop-2}, we also have the descent isomorphism
$\rHbm_{2d}(X/S,d)\xrightarrow{\sim}\invlim_{i\in\mbf{\Delta}}\rHbm_{2d}(X_i/S_i,d)$ of Abelian groups.
Combining everything together, we have a map
\begin{align*}
 \mr{tr}_f\colon z(X/S,d)&\cong
 \indlim_{S_\bullet\in\mr{HR}^{\mr{g}}(S)}
 \invlim_{i\in\mbf{\Delta}}\mc{Z}(-,d)\\
 &\rightarrow
 \indlim_{S_\bullet\in\mr{HR}^{\mr{g}}(S)}
 \invlim_{i\in\mbf{\Delta}}\rHbm_{2d}(-,d)
 \cong
 \rHbm_{2d}(X/S,d).
\end{align*}

\begin{lem}\label{functtrnai}
\leavevmode
 \begin{enumerate}
  \item\label{functtrnai-1}
       Consider the Cartesian diagram {\normalfont(\ref{stdcartdia})}.
       Assume $\dim(f^{-1}(y))\leq d$ for any point $y$ of $Y$, in
       which case the same property holds for $f'$. Then we have
       $g^*\circ\mr{tr}_f=\mr{tr}_{f'}\circ g'^*$.

  \item\label{functtrnai-3}
       Let $X\xrightarrow{g}X'\xrightarrow{f'}Y$ be morphisms and put $f:=f'\circ g$.
       We assume that for any $y\in Y$, $\dim(f^{(\prime)-1}(y))\leq d$ and $g$ is proper.
       Then we have $\mr{tr}_{f'}\circ g_*=g_*\circ\mr{tr}_{f}$.
 \end{enumerate}
\end{lem}
\begin{proof}
 Let us check the claim \eqref{functtrnai-1} of the lemma.
 Take a good $p$dh
 hypercovering $\alpha\colon Y_\bullet\rightarrow Y$. Then we are able to
 find a good $p$dh-hypercovering $\alpha'\colon Y'_\bullet\rightarrow Y'$
 which fits into the following diagram, not necessarily Cartesian:
 \begin{equation*}
  \xymatrix{
   Y'_\bullet\ar[r]\ar[d]_{\alpha'}&Y_\bullet\ar[d]^{\alpha}\\
  Y'\ar[r]&Y.
   }
 \end{equation*}
 Consider the following diagram:
 \begin{equation*}
  \xymatrix{
   z(X/Y,d)\ar[r]^-{\sim}\ar[d]&
   \invlim_{i\in\mbf{\Delta}}z(X_{Y_i}/Y_i,d)\ar[r]^-{\mr{Tr}}\ar[d]&
   \invlim_{i\in\mbf{\Delta}}\rHbm_{2d}(X_{Y_i}/Y_i,d)\ar[d]&
   \rHbm_{2d}(X/Y,d)\ar[l]_-{\sim}\ar[d]\\
  z(X'/Y',d)\ar[r]^-{\sim}&
   \invlim_{i\in\mbf{\Delta}}z(X'_{Y'_i}/Y'_i,d)\ar[r]^-{\mr{Tr}}&
   \invlim_{i\in\mbf{\Delta}}\rHbm_{2d}(X'_{Y'_i}/Y'_i,d)&
   \rHbm_{2d}(X'/Y',d).\ar[l]_-{\sim}
   }
 \end{equation*}
Both external squares are commutative by the functoriality of $z(-,d)$
 and $\rHbm(-,d)$, and the middle as well by
 \ref{bccommsim}-\eqref{bccommsim-2}.
 The claim \eqref{functtrnai-3} follows immediately from Lemma \ref{bccommsim}-\eqref{bccommsim-3}.
\end{proof}

\subsection{Proof of Theorem \ref{mainres}}\mbox{}\\
First, let us construct a morphism $z_{\mr{equi}}(-,d)\rightarrow\rHbm_{2d}(-,d)$.
Let $f\colon Y\rightarrow T$ be a morphism, and $w\in z_{\mr{equi}}(Y/T,d)$.
Let $W$ be the support of $w$, and $i\colon W\hookrightarrow Y$ be the closed immersion.
Then $w$ is the image of an element $w'\in z_{\mr{equi}}(W/T,d)$ via the morphism $i^z_*\colon z_{\mr{equi}}(W/T,d)\rightarrow z_{\mr{equi}}(Y/T,d)$.
Since $w\in z_{\mr{equi}}(Y/T,d)$, the dimension of each fiber of $f\circ i$ is $\leq d$.
Thus, we have already constructed the morphism $\mr{tr}_{f\circ i}\colon z_{\mr{equi}}(W/T,d)\rightarrow\rHbm_{2d}(W/T,d)$.
We define $\tau_{Y/T}(w):=i^{\mr{H}}_*\circ\mr{tr}_{f\circ i}(w')$,
where $i^{\mr{H}}_*\colon\rHbm_{2d}(W/T,d)\rightarrow\rHbm_{2d}(Y/T,d)$ is the pushforward.
This defines a map $\tau_{Y/Z}\colon z_{\mr{equi}}(Y/T,d)\rightarrow\rHbm_{2d}(Y/T,d)$.
In view of Lemma \ref{functtrnai}-\eqref{functtrnai-3}, this map is in fact a homomorphism of Abelian groups.
This map is compatible with base change and pushforward by Lemmas~\ref{functtrnai}-\eqref{functtrnai-1} and~\ref{functtrnai}-\eqref{functtrnai-3}.
The uniqueness of the map follows by Lemma~\ref{smcomp} and construction.

It remains to show the compatibility with respect to the product structure.
Let $X\xrightarrow{f}Y\xrightarrow{g}Z$ be morphisms, and $x\in z_{\mr{equi}}(X/Y,d)$, $y\in z_{\mr{equi}}(Y/Z,e)$.
By definition, we may assume that $Z$ is smooth, and $\mr{Supp}(x)\subset X$ and $\mr{Supp}(y)\subset Y$ are flat over $Z$.
By projection formula of bivariant theories (\emph{cf.}~\S\ref{FMbivdfn}), we may assume that $Y=\mr{Supp}(y)$ (with reduced induced scheme structure).
Then, by the compatibility with pushforward, we may replace $X$ by $\mr{Supp}(x)$.
In this situation, we are allowed to shrink $Z$ by its open dense subscheme because $\rHbm(X/Z,d+e)$ does not change by \S\ref{constmapsmo},
we may further assume that $y=[Y]$.
Now, for an open immersion $j\colon U\subset X$, we have restriction morphisms
$z_{\mr{equi}}(X/Z,n)\rightarrow z_{\mr{equi}}(U/Z,n)$ and $\rHbm(X/Z,n)\rightarrow\rHbm(U/Z,n)$
and we may check easily that these are compatible with $\tau_{X/Z}$.
Since $f\colon X\rightarrow Y$ is dominant, we may take open dense subschemes $U\subset X$ and $V\subset Y$ such that
$f(U)\subset V$, $U\rightarrow V$ is flat, and $V$ is smooth.
The compatibility with open immersion allows us to replace $X$ by $U$.
Since $\tau_{U/Y}(x)\bullet\tau_{Y/Z}(y)=\tau_{U/V}(x)\bullet\tau_{V/Z}(y|_V)$,
it suffices to show the claim for $U\rightarrow V\rightarrow Z$,
and in this case, we have already treated in Lemma~\ref{bccommsim}-\eqref{bccommsim-1} together with Lemma~\ref{smcomp}.
\qed

\section{$\infty$-enhancement of the trace map}
\label{sec5}
In this section, we upgrade the trace map to the $\infty$-categorical setting.

\subsection{}
Let $\Arr$ be the category of morphisms $X\rightarrow S$ in $\mr{Sch}_{/k}$
whose morphisms from $Y\rightarrow T$ to $X\rightarrow S$ consists of diagrams of the form
\begin{equation}
 \label{diaarr}
 \begin{gathered}
 \xymatrix{
  Y\ar[d]&X_T\ar[r]\ar[l]_{\alpha}\ar[d]\ar@{}[rd]|\square&
  X\ar[d]\\
 T\ar@{=}[r]&T\ar[r]^-{g}&S
  }
  \end{gathered}
\end{equation}
where $\alpha$ is proper.
The composition is defined in an evident manner, and we refer to \cite[\S5.2]{A} for the detail.
We often denote an object corresponding to $X\rightarrow S$ in $\Arr$ by $X/S$.
For $Y/T\in\Arr$, let $\mr{Cov}(Y/T)$ be the set of families $\{Y_{T_i}/T_i\rightarrow Y/T\}_{i\in I}$ where $\{T_i\rightarrow T\}$ is a cdh-covering.
The category $\Arr$ does not admits pullbacks in general, but each morphism $(Y_{T_i}/T_i)\rightarrow(Y/T)$ is {\it quarrable}, in other words,
for any morphism $(Y'/T')\rightarrow(Y/T)$, the pullback $(Y_{T_i}/T_i)\times_{(Y/T)}(Y'/T')\rightarrow(Y'/T')$ exists.
Indeed, we can check easily that $(Y_{T_i}/T_i)\times_{(Y'/T')}(Y/T)\cong(Y'\times_TT_i/T'\times_TT_i)$.
Thus, this family defines a pretopology on $\Arr$ in the sense of \cite[Expos\'e~II, \S1.3]{SGA}.

Now, fixing $(Y/T)\in\Arr$, we have the functor $\iota_{Y/T}\colon\mr{Sch}_{/T}\rightarrow\Arr$ sending $T'\rightarrow T$ to $(Y\times_TT'/T')$.
This functor commutes with pullbacks.
Putting the cdh-topology on $\mr{Sch}_{/T}$, the functor $\iota_{X/T}$ is cocontinuous (\emph{cf.}~\cite[Expos\'e~III, \S2.1]{SGA}) by \cite[Expos\'e~II, \S1.4]{SGA}.

\subsection{}
By associating the Abelian group $z(Y/T,n)$ to $Y\rightarrow T$, we have a functor $z^{\mr{SV}}(n)\colon\widetilde{\mr{Ar}}^{\mr{op}}\rightarrow\Sp^{\heartsuit}$.
Then $z^{\mr{SV}}(n)$ is an Abelian sheaf on $\Arr$.
Indeed, we must show the \v{C}ech descent with respect to the elements of $\mr{Cov}(Y/T)$ by \cite[Expos\'e~II, \S2.2]{SGA}.
This is exactly the contents of \cite[\S4.2.9]{SV}.
We define $z(n)$ to be the sheafification of $z^{\mr{SV}}(n)$ regarded as a spectra-valued presheaf on $\Arr$.

Now, by \cite[Lemma~C.3]{ES}, we have the following commutative diagram of geometric morphisms of $\infty$-topoi
\begin{equation}
 \label{cocontdia}
 \begin{gathered}
 \xymatrix@C=40pt{
  \Shv(\mr{Sch}_{/T,\mr{cdh}})\ar[r]^-{\iota^{\mr{s}}_{Y/T}}\ar[d]&\Shv(\widetilde{\mr{Ar}})\ar[d]\\
 \PShv(\mr{Sch}_{/T})\ar[r]^-{\iota_{Y/T}}&\PShv(\widetilde{\mr{Ar}}).
  }
  \end{gathered}
\end{equation}
Note that, since local objects (with respect to a localization) are stable under taking limits by definition,
$(\iota^{\mr{s}}_{Y/T})^*$ is commutes with limits by \cite[Lemma~C.3]{ES},
which justifies that $\iota^{\mr{s}}_{Y/T}$ is a geometric morphism.
Moreover, by \cite[Proposition 20.6.1.3]{SAG}, the functor $(\iota^{\mr{s}}_{Y/T})^*$ is given by composing with $\iota_{Y/T}$.
In particular, $z(n)\circ\iota_{Y/T}$ is the (cdh-)sheafification of $z^{\mr{SV}}(n)\circ\iota_{Y/T}$.

\subsection{}
Assume we are given a morphism $F\colon(Y/T)\rightarrow(X/S)$ in $\Arr$ as in (\ref{diaarr}).
Then we have the morphism of spectra
\begin{equation*}
 F^*\colon\rHbm(X/S,d)[-2d]\xrightarrow{\ g^*\ }\rHbm(X_T/T,d)[-2d]\xrightarrow{\ \alpha_*\ }\rHbm(Y/T,d)[-2d].
\end{equation*}
With this morphism, we can check easily that the association $\rHbm(X/S,d)[-2d]$ to $X/S\in\Arr$ yields a functor $H\colon\Arr^{\mr{op}}\rightarrow\mr{h}\Sp$.
It is natural to expect that this morphism can be lifted to a functor of $\infty$-categories $\Arr^{\mr{op}}\rightarrow\Sp$.
We put the existence as an assumption as follows:

\begin{quote}
 Assume we are given a functor $\sHbm(d)\colon\Arr^{\mr{op}}\rightarrow\Sp$ between $\infty$-categories
 whose induced functor between homotopy categories coincides with $H$ above.
%
%
\end{quote}
We constructed such a functor in \cite[Example 6.8]{A}, and also in \cite[\S C.3]{A2} using a slightly different method.
Now, we have the following $\infty$-enhancement of the trace map.

\begin{thm}
 There exists essentially uniquely a morphism of spectra-valued sheaves $\tau^{\dag}\colon z(d)\rightarrow\sHbm(d)$
 on $\Arr$ for any $d$ such that the composition
 \begin{equation*}
  z_{\mr{equi}}(-,d)\lhook\joinrel\longrightarrow
   z(-,d)
   \cong
   \pi_0z(d)
   \xrightarrow{\ \pi_0(\tau^{\dag})\ }
   \pi_0\sHbm(d)
   \cong
   \rHbm_{2d}(-,d)
 \end{equation*}
 coincides with the morphism $\tau$ of Theorem {\normalfont\ref{mainres}}.
\end{thm}

\begin{proof}
 Let $\pi_0z_{\mr{equi}}(d)\subset\pi_0z(d)$ be the subsheaf so that the value at $X\rightarrow S$ is $z_{\mr{equi}}(X/S,d)$.
 Note that $\pi_0z_{\mr{equi}}(d)$ is just a notation and {\em not} $\pi_0$ of some presheaf $z_{\mr{equi}}(d)$.
 We first define the trace map for $\pi_0z_{\mr{equi}}(d)$.
 Let $\Arr_d$ be the full subcategory of $\Arr$ consisting of objects $f\colon X\rightarrow S$ such that $\dim(f)\leq d$.
 First, let us construct the map after restricting to $\Arr_d$.
 We have already constructed the map of spectra-valued presheaves
 \begin{equation}
  \label{mapforpa}
   \pi_0z_{\mr{equi}}(d)|_{\Arr_d}
   \xrightarrow{\ \tau\ }
   \pi_0\sHbm(d)|_{\Arr_d}
   \xleftarrow{\ \sim\ }
   \tau_{\geq0}\sHbm(d)|_{\Arr_d}
   \longrightarrow\sHbm(d)|_{\Arr_d}.
 \end{equation}
 Here, the equivalence follows by Lemma \ref{sHbmprop}-\eqref{sHbmprop-2} since we are restricting the functor to $\Arr_d$.
 Now, let the category $\App$ be the full subcategory of $\mr{Fun}(\Delta^1,\Arr)$ spanned by the morphisms $h\colon(X/S)\rightarrow(Y/T)$ in $\Arr$
 such that $(Y/T)$ belongs to $\Arr_d$.
 We have functors $s,t\colon\App\rightarrow\Arr$ where $s$ is the evaluation at $\{0\}\in\Delta^1$, and $t$ is at $\{1\}$.
 Namely, for $h$ above, we have $s(h)=(X/S)$ and $t(h)=(Y/T)$.
 By \cite[Corollary~2.4.7.12]{HTT}, $s$ is a Cartesian fibration.
 Note that we have the natural transform $s\rightarrow t$ and this induces the morphism of functors
 $\phi\colon\mc{F}\circ t^{\mr{op}}\rightarrow\mc{F}\circ s^{\mr{op}}$ for any functor $\mc{F}\colon\Arr^{\mr{op}}\rightarrow\Sp$.
 From now on, we abbreviate $\mc{F}\circ t^{\mr{op}}$, $\mc{F}\circ s^{\mr{op}}$ by $\mc{F}\circ t$, $\mc{F}\circ s$ to avoid heavy notations.
 By \eqref{mapforpa}, we have the map $\pi_0z_{\mr{equi}}(d)\circ t\rightarrow\sHbm(d)\circ t$ of spectra-valued presheaves on $\App$.
 Now, we have the following diagram of $\infty$-categories
 \begin{equation*}
  \xymatrix{
   \App^{\mr{op}}\ar[d]_-{s^{\mr{op}}}\ar[r]^-{F}&\Sp\ar[d]\\
  \Arr^{\mr{op}}\ar[r]&\Delta^0,&
   }
 \end{equation*}
 where $F$ is either $\pi_0z_{\mr{equi}}(d)\circ t$ or $\sHbm(d)\circ t$.
 Since $s$ is a Cartesian fibration, $s^{\mr{op}}$ is a coCartesian fibration.
 Since the $\infty$-category $\Sp$ is presentable, any left Kan extension exists by \cite[Proposition 4.3.2.15]{HTT}.
 We denote by $\mr{LKE}(F)\colon\Arr^{\mr{op}}\rightarrow\Sp$ a left Kan extension of the above diagram.
 We have the following diagram of spectra-valued presheaves:
 \begin{equation*}
  \xymatrix{
  \mr{LKE}(\pi_0z_{\mr{equi}}(d)\circ t)\ar[r]\ar[d]&
  \mr{LKE}(\sHbm(d)\circ t)\ar[d]\\
  \pi_0z_{\mr{equi}}(d)&\sHbm(d).
  }
 \end{equation*}
 Here, the vertical morphisms are defined by taking the adjoint to $\phi$.
 We claim that the left vertical map is equivalent.
 For this, it suffices to show that $\pi_0z_{\mr{equi}}(d)$ is in fact a left Kan extension of $\pi_0z_{\mr{equi}}(d)\circ t$.
 Let $(X/S)\in\Arr$, and we denote by $\App_{X/S}$ the fiber of $s$ over $(X/S)$.
 Since $s^{\mr{op}}$ is a coCartesian fibration, by invoking \cite[Proposition 4.3.3.10]{HTT}, it suffices to show that $\bigl(\pi_0z_{\mr{equi}}(d)\bigr)(X/S)$
 is a left Kan extension of $\bigl(\pi_0z_{\mr{equi}}(d)\circ t\bigr)|_{\App^{\mr{op}}_{X/S}}$ along the canonical map $\App^{\mr{op}}_{X/S}\rightarrow{\{X/S\}}$ for any $X/S$.
 This amounts to showing that the morphism of {\em spectra} (namely the colimit is the ``derived colimit'')
 \begin{equation*}
  \indlim_{D\in\App^{\mr{op}}_{X/S}}z_{\mr{equi}}\bigl(t(D)/S,d\bigr)
   \longrightarrow
   z_{\mr{equi}}(X/S,d)
 \end{equation*}
 is an equivalence. 
 Let $C(X/S)$ be the category of closed immersions $Z\hookrightarrow X$ such that the composition $Z\rightarrow S$ is in $\Arr_d$.
 The category $\App^{\mr{op}}_{X/S}$ is filtered and the inclusion $C(X/S)\hookrightarrow\App^{\mr{op}}_{X/S}$ cofinal.
 Thus, the colimit is t-exact by \cite[Proposition 1.3.2.7]{SAG}, and it suffices to show the morphism
 $\indlim_{Z\in C(X/S)}z_{\mr{equi}}\bigl(Z/S,d\bigr)\rightarrow z_{\mr{equi}}(X/S,d)$ of {\em Abelian groups} is an isomorphism.
 This follows by definition.
 Thus we have the map $\pi_0z_{\mr{equi}}(d)\rightarrow\sHbm(d)$ of spectra-valued presheaves.

 Finally, let us extend this map to the required map.
 The $\infty$-presheaf $\sHbm(d)$ on $\Arr$ is in fact an $\infty$-sheaf.
 Indeed, let $\sHbm(d)\rightarrow L\sHbm(d)$ be the localization morphism.
 We must show that this morphism is an equivalence.
 Recall that for an $\infty$-category $\mc{C}$, a simplicial set $S$, and a morphism $\alpha\colon F\rightarrow G$ in $\mr{Fun}(S,\mc{C})$,
 $\alpha$ is an equivalence if and only if $\alpha(s)$ is an equivalence for any vertex $s$ of $S$.
 We believe that this is well-known, but a (fairly) indirectly way to see this is by applying \cite[Corollary 5.1.2.3]{HTT}
 to the diagram $(\Delta^0)^{\triangleright}\rightarrow\mr{Fun}(S,\mc{C})$ given by $\alpha$.
 Now, let $(Y/T)\in\Arr$.
 Since the verification is pointwise by the recalled fact,
 it suffices to show that $\sHbm(d)\circ\iota_{Y/T}\rightarrow(L\sHbm(d))\circ\iota_{Y/T}$ is an equivalence.
 By \eqref{cocontdia}, this morphism can be identified with $\iota^*_{Y/T}\sHbm(d)\rightarrow L\bigl(\iota_{Y/T}^*\sHbm(d)\bigr)$,
 which is an equivalence since $\iota^*_{X/T}\sHbm(d)\simeq\ulrHbm(X/S,d)$ is already a cdh-sheaf (\emph{cf.} Lemma~\ref{sHbmprop}-\eqref{sHbmprop-1}).
 Thus, by taking the sheafification to the morphism $\pi_0z_{\mr{equi}}(d)\rightarrow\sHbm(d)$,
 we get the morphism $z(d)\rightarrow\sHbm(d)$.
 The essential uniqueness follows by construction, and the detail is left to the reader.
\end{proof}


\end{document}